\documentclass[12pt,draft]{amsart}

\usepackage[all]{xy}
\usepackage{amsfonts}
\usepackage{amssymb}

\usepackage{enumerate}

\newtheorem{teo}{Theorem}[section]
\newtheorem{lem}[teo]{Lemma}
\newtheorem{prop}[teo]{Proposition}

\newtheorem{cor}[teo]{Corollary}

\theoremstyle{definition}
\newtheorem{dfn}[teo]{Definition}
\newtheorem{rk}[teo]{Remark}

\def\<{\langle}
\def\>{\rangle}

\def\a{\alpha}

\def\d{\delta}
\def\e{\varepsilon}

\def\f{{\varphi}}

\def\F{{\Phi}}

\def\bK{{\mathbf{K}}}
\def\bBK{{\mathbf{BK}}}
\def\bL{{\mathbf{L}}}

\def\A{{\mathcal A}}

\def\M{{\mathcal M}}
\def\cN{{\mathcal N}}
\def\Hom{\mathop{\rm Hom}\nolimits}
\def\End{\mathop{\rm End}\nolimits}
\def\Ker{\mathop{\rm Ker}\nolimits}
\def\Im{\mathop{\rm Im}\nolimits}

\def\dist{\mathop{\rm dist}\nolimits}

\def\Span{\mathop{\rm Span}\nolimits}

\def\1{\mathbf 1}

\title[Banach-compact operators, precompactness and frames]{Banach-compact operators, $\A$-precompactness, and frames in Hilbert $C^*$-modules}
\author{Denis Fufaev}
\author{Evgenij Troitsky}
\thanks{The work was supported by the grant 24-11-00124 of the Russian Science Foundation.}
\address{Moscow Center for Fundamental and Applied Mathematics,
Dept. Mech. and Math., 
	Lomonosov Moscow State University, 119991 Moscow, Russia}
\email{denis.fufaev@math.msu.ru, fufaevdv@rambler.ru}
\email{troitsky@mech.math.msu.su}


\begin{document}
\begin{abstract}
For a pair of Hilbert $C^*$-modules $\M$, $\cN$ over a $C^*$-algebra $\A$, there are two notions of ``rank-one operators'':
$\theta_{x,y}:\M\to\cN$, $\theta_{x,y}(z)=x\<y,z\>$, where $y,z\in\M$, $x\in\cN$ (these operators are called elementary $\A$-compact or elementary Kasparov operators) and 
$\theta_{x,f}:\M\to\cN$, $\theta_{x,f}(z)=xf(z)$, where $z\in\M$, $x\in\cN$, and $f$ is a bounded $\A$-functional on $\M$ (introduced by Manuilov). They generate a $C^*$-bimodule $\bK(\M,\cN)$ ($\A$-compact operators) over the $C^*$-algebras of adjointable operators and a Banach bimodule $\bBK(\M,\cN)$ (Banach-compact operators)
over the algebras of all bounded morphisms, respectively.

In order to give a geometric characterization of these classes of operators, we consider the notion of $\A$-precompactness (developing the one introduced by Manuilov). Banach-compact operators can be characterized as those with $\A$-precompact image of the unit ball. Another characterization obtained here is in terms of total boundedness of this set with respect to the uniform structure introduced by one of us previously.
The constructions and proofs turn out to be closely related to the concept of frames in a Hilbert $C^*$-module.
\end{abstract}

\maketitle

\section{Introduction}

In the theory of Hilbert spaces, a well-known criterion states that an operator is the norm limit of finite-rank operators if and only if the image of the unit ball is precompact, i.e. totally bounded with respect to the norm. This criterion fails for Hilbert $C^*$-modules: for any infinite-dimensional unital $C^*$-algebra $\A$, the identity operator on $\A$ is of finite rank (in fact, rank one), yet the unit ball is not totally bounded in norm. Consequently, a natural question arises: can the approximability of operators by finite-rank operators on Hilbert $C^*$-modules be characterized via the geometric properties of the image of the unit ball?
Note that in the module context one must distinguish between two notions of finite-rank operators: those defined via arbitrary bounded $\A$-linear functionals (which lead us to the definition of Banach-compact operators) and a more restrictive class defined via inner products with module elements (which lead us to the definition of $\A$-compact operators), see Definition \ref{dfn:Acompact} below.
Compact-type operators and related classes have been investigated in various
settings and from several points of view; see, for example,
\cite{Magajna1998JLMS,WangWeiLiu2019JMAA}.

In \cite{KeckicLazovic2018} and \cite{lazovic2018}, certain uniform structures, given by families of seminorms on the range module, were introduced and used to obtain results on $\A$-compact operators for particular classes of modules. In \cite{Troitsky2020JMAA}, a new system of seminorms was proposed, and for operators with countably generated range a criterion for $\A$-compactness of bounded adjointable operators was established (namely, \cite[Theorem 2.13]{Troitsky2020JMAA}). This was extended in \cite{TroitFuf2020} to modules admitting a standard frame, equivalently, to modules that are direct summands of standard modules. In the paper \cite{FufTro2024UniformStr}, more general versions of the corresponding uniform structure were constructed, and the same equivalence was proved for countably generated modules.

However, in \cite{Fuf2021faa, Fuf2022Path, Fuf2023, Fuf2025} it was shown that such an equivalence fails in greater generality (i.e. for modules without standard frames). This failure already occurs for the simplest example of a module, namely a commutative $C^*$-algebra considered as a module over itself. In that setting, an appropriate classification of topological spaces was obtained, and it was shown that the criterion does not hold for certain classes of spaces and the corresponding algebras. Several examples were constructed, and part of the theory was extended to the noncommutative setting.

A parallel approach was developed for the broader class of Banach-compact operators, where $\A$-compact operators appear precisely as the adjointable ones. In \cite[Proposition 2.3]{m_bratisl}, the classical idea of compactness as an approximation by finite-dimensional subspaces was generalized: an $\A$-precompact subset of a standard module over a unital $C^*$-algebra $\A$ was defined as one that can be norm-approximated by free finitely generated submodules. In this setting, it was proved that an endomorphism of a standard module is Banach-compact if and only if the image of the unit ball is $\A$-precompact.
Nevertheless, the requirement that the approximating submodules are free is overly restrictive for arbitrary modules. The purpose of this paper is to overcome this limitation and obtain an analogous criterion for Hilbert $C^*$-modules of general type.

The above-mentioned results were further developed in the works \cite{KeckicLazovic2019} and \cite{KeckicLazovic2026}, where D.~Ke{\v{c}}ki{\'c} and Z.~Lazovi{\'c} in fact reformulated the concept of $\A$-precompactness in terms of measures of noncompactness, i.e. functions of sets that vanish on sets that are precompact with respect to some geometric structure. In \cite{KeckicLazovic2019}, as in \cite{m_bratisl}, they worked with the standard module, whereas in \cite{KeckicLazovic2026} they extended the results to countably generated modules. In particular, they prove that in such modules a set is totally bounded with respect to the uniform structure introduced in \cite{Troitsky2020JMAA} if and only if the specific measure of noncompactness equals zero. See also Remarks \ref{rk:lambda}, \ref{rk:from_count}, \ref{rk:partcase} below.

To work with general modules we also must abandon the condition that the generators of the approximating submodule are free. Instead, we also use the condition that the coefficients in the linear combinations of generators approximating the considered set are uniformly bounded on this set (for fixed $\e>0$), as in \cite{KeckicLazovic2026}.
This restriction yields a property that allows us to prove the Banach-compactness criterion for an arbitrary module. This property also turns out to be equivalent to the analogue of another characterization of precompact sets in Hilbert spaces: a set is precompact if and only if the identity operator can be uniformly approximated by operators of finite rank on this set.

 The main results are the following: we prove that uniformly bounded coefficients imply the uniform convergence in the reconstruction formula (Proposition \ref{prop:compcount}), describe the notion of $\A$-precompact subsets for general Hilbert $C^*$-modules (Theorem \ref{thm:setcomp})  and prove that an operator is Banach-compact if and only if the image of the unit ball is $\A$-precompact (Theorem \ref{thm:opercomp}). Then we prove that this notion of $\A$-precompactness is indeed a generalization of the old one (Theorem \ref{thm_old}), compare this result with those obtained via uniform structures (Theorem \ref{thm:totb_comp} and Corollary \ref{cor:totb_comp}) and study an example showing that even in the case of a standard module over a unital algebra the condition that a subset can be approximated by finitely generated submodules is insufficient (Theorem \ref{ex_bad}). We also obtain that any Banach-compact operator can be represented by a norm-convergent series of finite Banach-rank operators (Corollary \ref{cor_series}). 

\section{Preliminaries}

We briefly recall the notions of Hilbert $C^*$-modules, Banach-compact and $\A$-compact operators, and standard frames that will be used. At the end of the section we also recall the uniform structure from the earlier approach, which will be needed for comparison in Theorems \ref{thm:totb_comp} and \ref{cor:totb_comp}.

The following basic notions and facts can be found in
\cite{Lance,MTBook,ManuilovTroit2000JMS}.

\begin{dfn}
\rm
A (right) pre-Hilbert $C^*$-module over a $C^*$-algebra $\A$
is an $\A$-module equipped with an $\A$-\emph{inner product}
$\<.,.\>:\M\times\M\to \A$, which is a sesquilinear form on the
underlying linear space such that, for any $x,y\in\M$, $a\in\A$:
\begin{enumerate}
\item $\<x,x\> \ge 0$;
\item $\<x,x\> = 0$ if and only if $x=0$;
\item $\<y,x\>=\<x,y\>^*$;
\item $\<x,y\cdot a\>=\<x,y\>a$.
\end{enumerate}

A \emph{Hilbert $C^*$-module} is a pre-Hilbert $C^*$-module over $\A$  which 
is complete with respect to its norm $\|x\|=\|\<x,x\>\|^{1/2}$.

A pre-Hilbert $C^*$-module $\M$ is called \emph{countably generated} if there is a countable collection of its elements such that their $\A$-linear combinations are dense in $\M$.

The \emph{Hilbert sum} of Hilbert $C^*$-modules in the evident sense will be denoted by $\oplus$.

Denote by $\ell^J_2(\A)$ the ``standard module of cardinality $J$'', i.e. the module $\bigoplus\limits_{j\in J}\A$, where $\A$ is considered as an $\A$-module with inner product $\<a,b\>=a^*b$.
When $J=\mathbb N$, i.e. there is a countable set of copies of $\A$, it is denoted by $\ell_2(\A)$.

Also define the \emph{dual module} $\M'=\Hom_\A(\M,\A)$ as the set of all bounded $\A$-linear functionals, which is equipped with the right $\A$-module structure $(fa)(\cdot)=a^*f(\cdot)$, $f\in\M'$, $a\in\A$.
\end{dfn}

We will say that a submodule $L$ of a Hilbert $C^*$-module $\M$ over $\A$ is a free submodule if $L\cong \A^n$ for some $n\in\mathbb N$. When $\A$ is unital, one can prove that this means that there exist $e_1,\dots,e_n\in L$ such that $\<e_i,e_j\>=\d_{i,j}\cdot1_{\A}$ and $L=\Span_\A(e_1,\dots,e_n)$.

For any $C^*$-algebra $\A$, we denote by $\dot{\A}$ its unitalization, which we assume to be equal to $\A$ if $\A$ is unital.

\begin{dfn}\label{dfn:operators}
\rm
By an \emph{operator} we mean a bounded $\A$-linear map.
An operator having an adjoint (in an evident sense) is 
\emph{adjointable} (see \cite[Section 2.1]{MTBook}).
We will denote the Banach space of all operators
$F: \M\to \cN$ by  ${\mathbf{L}}(\M,\cN)$
and the Banach space of adjointable operators
by ${\mathbf{L}}^*(\M,\cN)$.  
\end{dfn}

\begin{dfn}\label{dfn:Acompact}
An \emph{elementary} \emph{Banach-compact} operator
$\theta_{x,f}:\M\to\cN$, where $x\in\cN$ 
and $f\in\M'$, is defined as $\theta_{x,f}(z):=x\cdot f(z)$.
Then the Banach space $\bBK(\M,\cN)$
of \emph{Banach-compact} operators
is the closure of the subspace generated by all 
elementary Banach-compact operators in $\bL(\M,\cN)$.
If we take only functionals of the form $\<y,z\>$ for some $y\in\M$, then we obtain the definitions of an elementary $\A$-compact operator and the corresponding Banach space $\bK(\M,\cN)$ of $\A$-compact operators.

Finite sums of elementary Banach-compact operators
\(\theta_{x,f}\) will be called \emph{finite Banach-rank operators},
while finite sums of elementary \(\A\)-compact operators
\(\theta_{x,y}\) will be called \emph{finite \(\A\)-rank operators.}
\end{dfn}

Note that $T\theta_{x,f}=\theta_{Tx,f}$ if $T\in \bL(\cN_1,\cN_2)$,
and $\theta_{x,f}S=\theta_{x,f\circ S}$ if $S\in \bL(\M_1,\M_2)$.
In particular, this implies that $\bBK(\M,\M)$ is a (closed two-sided) ideal in $\bL(\M,\M)$.
 
\begin{dfn}\label{dfn:fr}
Let $\cN$ be a Hilbert $C^*$-module over a $C^*$-algebra $\A$ and $J$ be some index set. A family $\{x_j\}_{j\in J}$ of elements of $\cN$ is said to be a \emph{standard frame} in $\cN$ if there exist positive constants $c_1, c_2$ such that for any $x\in\cN$ the series $\sum\limits_j \<x,x_j\>\<x_j,x\>$
converges in norm in $\A$ (in particular, only countably many terms of the series are nonzero; also, since all summands are positive, the sum is independent of the order of summation, that is, the series converges unconditionally) and the following inequalities hold:
$$c_1\<x,x\>\le \sum\limits_j \<x,x_j\>\<x_j,x\>\le c_2\<x,x\>.$$
\end{dfn}

\begin{rk}\label{rk:stab_fr}
One of the most important structural results of the frame theory is the following: from the results of Frank and Larson (Example 3.5, Theorems 5.3 and 4.1 in \cite{FrankLarson2002}, see also Theorem 1.1 in \cite{HLi2010}) it follows that a Hilbert $C^*$-module
$\cN$ over a $C^*$-algebra $\A$ can be represented as an orthogonal direct summand in the standard module of some cardinality $\ell^J_2(\dot{\A})=\bigoplus\limits_{j\in J}\dot{\A}$ over the unitalization algebra if and only if there exists a standard frame in $\cN$.
\end{rk}

If a Hilbert $C^*$-module $\cN$ over a $C^*$-algebra $\A$ is countably generated, then it always has a frame. Indeed, in this case $\cN$ is also countably generated as an $\dot{\A}$-module and by the Kasparov stabilization theorem (\cite[Theorem 2]{Kasp}, see also \cite[Theorem 1.4.2]{MTBook}) there exists an isomorphism $\cN\oplus\ell_2^{}(\dot{\A})\cong \ell_2^{}(\dot{\A})$, so $\cN$ has a countable frame as an $\dot{\A}$-module which obviously is also a frame in $\cN$ as an ${\A}$-module.

For a standard frame $\{x_j\}_{j\in J}$ the \emph{analysis operator} $\Theta:\cN\to\ell_2^J(\A)$ and its adjoint $\Theta^*$ are defined by the formulas 
\begin{equation}\label{eq:ThetaAndTheta*}
\Theta(x)=(\<x_j,x\>)_{j\in J}\qquad \mbox{ and }\qquad \Theta^*((a_j)_{j\in J})=\sum\limits_{j\in J}x_j a_j    
\end{equation}
 (see \cite[Sections 4 and 6]{FrankLarson2002} for details). We also denote the frame operator by $S=\Theta^*\Theta$. It is self-adjoint and invertible since $c_1\<x,x\>\le \<\Theta x,\Theta x\>=\<\Theta^* \Theta x, x\>$.

The \emph{canonical dual frame} is defined by the formula $g_j=(\Theta^*\Theta)^{-1}x_j=S^{-1}x_j$.

\begin{rk}\label{rk:fr_appr}

For any element $x\in\cN$ there exists an at most countable subset $\mathfrak{A}_0\subset J$ such that $\<x,x_j\>=0$ for $j\notin\mathfrak{A}_0$. Then there exists an at most countable $\mathfrak{A}_1\subset J$ such that $\<x_k,x_j\>=0$ for $j\notin\mathfrak{A}_1$ and $k\in\mathfrak{A}_0$, and so on. 
Thus, we obtain a countably generated submodule $L=\overline{\Span_\A(\{x\}\cup\{x_j\}_{j\in\mathfrak A})}$, where $\mathfrak A=\bigcup\limits_{l\ge0}\mathfrak A_l$, and $\{x_j\}_{j\in\mathfrak A}$ is an at most countable standard frame of $L$ since for any $y\in L$

$$
\sum_{j\in \mathfrak A} \<y,x_j\> \<x_j,y\>=\sum_{j\in J} \<y,x_j\> \<x_j,y\>,
$$
and for the frame operator $S_{\mathfrak{A}}$ of $\{x_j\}_{j\in\mathfrak A}$ we have $S_{\mathfrak{A}}(y)=\sum\limits_{j\in \mathfrak{A}} x_j \<x_j,y\>=\sum\limits_{j\in J} x_j \<x_j,y\>=S|_{L}(y)$.

Also the operator $S_{\mathfrak A}:L\to L$ is invertible, so the restriction of $S$ to $L$, $S|_L:L\to L$, is invertible as well, hence $S^{-1}(y)\in L$ for $y\in L$. 
So, for $j\notin\mathfrak A$ and $y\in L$ we have $\<g_j,y\>=\<x_j,S^{-1}y\>=0$.
In addition, for $j\in\mathfrak A$ we have $x_j\in L$, so $S_{\mathfrak A}^{-1}(x_j)=S_{}^{-1}(x_j)=g_j$, and by \cite[Theorem 6.1]{FrankLarson2002} the reconstruction formula 
$$
y=\sum\limits_{j\in\mathfrak A}x_j\<g_j,y\>=\sum\limits_{j\in{J}}x_j\<g_j,y\>
$$
holds. In particular, any standard frame is a set of generators for the Hilbert \(C^*\)-module in which it is a frame.

\end{rk}

We will use the following simple observation.

\begin{lem}\label{lem:countcountframe}
    Let $\M$ have a countable standard frame $x_i$. Then any other standard frame $\{y_\a\}$,
    $\a\in \mathfrak{A}$, of $\M$ is also countable
    (more precisely, has at most countably many nonzero elements).
\end{lem}

\begin{proof}
    Each $x_i$ is orthogonal to all $y_\a$ except a countable set $\mathfrak{A}_i$.
    Denote by $\mathfrak{A}'$ the countable set given by the union of all $\mathfrak{A}_i$, $i\in \mathbb N$.
    For $\a\notin \mathfrak{A}'$ and any $i$, we have $\<x_i,y_\a\>=0$, and hence $y_\a=0$. Indeed, using the definition of a standard frame, we have
    $
    c_1\<y_\a,y_\a\>\le\sum_{i\in \mathbb N} \<y_\a,x_i\> \<x_i,y_\a\>=0
    $, where $c_1$ is the lower frame constant of $\{x_n\}_{n\in\mathbb N}$.
    \end{proof}

We recall the previously obtained results describing $\A$-compact operators in order to compare them later with the results established in the present article.
The uniform structures under consideration on submodules of $\cN$ are defined as follows (see \cite{Fuf2021faa} and \cite{Troitsky2020JMAA} for details).

\begin{dfn}\label{dfn:admissyst}
\rm
Consider a Hilbert $C^*$-module $\cN$ over $\A$. A countable system $X=\{x_{i}\}$ of elements of $\cN$ is called \emph{admissible} for a (possibly not closed) submodule $\cN^0\subseteq \cN$
(or $\cN^0$-\emph{admissible}) if
\begin{enumerate}
\item[1)] 
 for each $x\in\cN^0$
the series $\sum_i \<x,x_i\>\<x_i,x\>$ is convergent in $\A$;
\item[2)]
its sum is bounded by $\<x,x\>$, that is, $\sum_i \<x,x_i\>\<x_i,x\> \le \<x,x\>$; 
\item[3)]$\|x_i\|\le 1$ for any $i$.
\end{enumerate}
\end{dfn}

\begin{dfn}\rm
Denote by $\F$ a countable collection $\{\f_1,\f_2,\dots\}$
of states on $\A$ (i.e. positive linear functionals of norm 1). For each pair $(X,\F)$ 
with an $\cN^0$-admissible $X$,
consider a nonnegative function defined by
the equality
$$
\nu_{X,\F}(x)^2:=\sup_k 
\sum_{i=k}^\infty |\f_k\left(\<x,x_i\>\right)|^2,\quad x\in \cN^0. 
$$
It can be checked that this is a seminorm on the module $\cN^0$.
Denote the system of all these functions by 
$\mathbb{SN}(\cN,\cN^0)$. We will also write $(X,\F)\in 
\mathbb{A}(\cN,\cN^0)$
for pairs with admissible $X$. 
\end{dfn}

\begin{dfn}\label{dfn:pseme}
For $(X,\F)\in \mathbb{A}(\cN,\cN^0)$ consider
the following function
$d_{X,\F}:\cN^0\times \cN^0\to [0,+\infty)$
$$
d_{X,\F}(x,y)^2:=\nu_{X,\F}(x-y)^2=
\sup_k 
\sum_{i=k}^\infty |\f_k\left(\<x-y,x_i\>\right)|^2,\quad x,y\in \cN^0. 
$$
We will write $d_{X,\F}\in \mathbb{PM}(\cN,\cN^0)$. 
\end{dfn}

Evidently, the functions
$d_{X,\F}$ are \emph{pseudometrics} in the sense of  \cite[Definition 2.10]{Troitsky2020JMAA} (and \cite[Chapter IX, \S 1]{BourbakiTop2}), so they form a uniform structure.

The definition of \emph{totally bounded} sets 
for the uniform structure under consideration
(or for the system $\mathbb{PM}(\cN,\cN^0)$) takes the following form.

\begin{dfn}\label{dfn:totbaundset}
\rm
A set $Y\subseteq \cN^0 \subseteq  \cN$ is \emph{totally bounded}
with respect to this
uniform structure if for any $(X,\F)$, 
where $X \subseteq  \cN$ is $\cN^0$-admissible,
and any 
$\e>0$ there exists a finite collection $y_1,\dots,y_n$
of elements of $Y$ such that the sets
$$
\left\{ y\in Y\,|\, d_{X,\F}(y_i,y)<\e\right\}
$$  
form a cover of $Y$. This finite collection is called an 
$\e$\emph{-net in $Y$ for} $d_{X,\F}$.

If so, we will say that $Y$ is $(\cN,\cN^0)$-\emph{totally bounded}.
\end{dfn}

In these terms the $\A$-compactness of operators for some class of modules can be described as follows (\cite[Theorem 3.5]{TroitFuf2020}):

\begin{teo}
Suppose that $\M$ and $\cN$ are Hilbert $\A$-modules, $\cN$ has a standard frame,
$F:\M\to\cN$ is an adjointable operator
and $F(B)$ is $(\cN,\cN)$-totally bounded, where $B$ is the unit ball of $\M$.
Then $F$
is $\A$-compact as an operator from $\M$ to $\cN$.
\end{teo}

The converse holds for arbitrary modules \cite[Theorem 2.4]{TroitFuf2020}.  

The following important property follows from \cite[Lemma 1.16]{FufTro2024UniformStr}.

\begin{lem}\label{lem:directsum_totbu} 
Suppose that a set $Y\subseteq  \cN=\cN_1\oplus \cN_2$ 
is $(\cN,\cN^0)$-totally bounded, 
where $\cN^0$ is a submodule. Then
$p_1 Y$ and $p_2 Y$ are 
$(\cN_1,\cN^0_1)$- and $(\cN_2,\cN^0_2)$-totally bounded, 
respectively, where $\cN^0_1=p_1(\cN^0)$,
$\cN^0_2=p_2(\cN^0)$,
and $p_1: \cN_1\oplus \cN_2 \to \cN_1$, $p_2: \cN_1\oplus \cN_2 \to \cN_2$ are the orthogonal projections.

Conversely, if $p_1 Y$ and $p_2 Y$ are 
$(\cN_1,\cN^0_1)$- and $(\cN_2,\cN^0_2)$-totally bounded
for some submodules $\cN^0_1$ and $\cN^0_2$, 
respectively, then $Y$ is $(\cN,\cN^0_1\oplus\cN^0_2)$-totally
bounded.
\end{lem}

\section{Technical lemmas}

We need a more subtle property of standard frames that is well known in the context of Hilbert spaces.

\begin{lem}\label{lem:partsum}
Suppose that $\cN$ is a Hilbert $\A$-module and $\{x_j\}_{j\in J}$ is a standard frame in $\cN$ with frame constants $c_1$ and $c_2$ and canonical dual frame $\{g_j\}_{j\in J}$.
Then for an operator corresponding to a finite partial sum of the reconstruction formula, 
$P_{J'}(x)=\sum\limits_{j\in{J'}}x_j\<g_j,x\>$, where $J'\subset J$ is finite, we have an estimate 
$||P_{J'}||\le\frac{c_2}{c_1}$.
\end{lem}

\begin{proof}
Since $\<\Theta x,\Theta x\>= \sum\limits_j \<x,x_j\>\<x_j,x\>\le c_2\<x,x\>$, by \cite[Corollary 2.1.6]{MTBook} we have $||\Theta||=||\Theta^*||\le\sqrt{c_2}$.

By \cite[Lemma 4.1]{Lance}, the inequality $c_1\<x,x\>\le \<\Theta x,\Theta x\>=\<\Theta^* \Theta x, x\>$ 
implies that $\Theta^*\Theta\ge c_1 \operatorname{Id}$ (in the $C^*$-algebra $\End^*_\A(\cN)$ of adjointable endomorphisms of $\cN$). Hence there exists $(\Theta^*\Theta)^{-1}$ and $||(\Theta^*\Theta)^{-1}||\le\frac{1}{c_1}$.

For any finite subset $J'\subset J$ the orthogonal projection $\pi_{J'}:\ell_2^J(\A)\to \ell_2^J(\A)$ is defined. Using (\ref{eq:ThetaAndTheta*}), we have for the finite partial sum 
\begin{eqnarray*}
P_{J'}(x)&=&\sum\limits_{j\in{J'}}x_j\<g_j,x\>=\sum\limits_{j\in{J'}}x_j\<(\Theta^*\Theta)^{-1}x_j,x\>=\sum\limits_{j\in{J'}}x_j\<x_j,(\Theta^*\Theta)^{-1}x\> \\
&=&\Theta^*\circ\pi_{J'}\circ\Theta\circ(\Theta^*\Theta)^{-1}(x).   
\end{eqnarray*}
Hence
$||P_{J'}||\le||\Theta^*||\cdot||\pi_{J'}||\cdot||\Theta||\cdot||(\Theta^*\Theta)^{-1}||\le\frac{c_2}{c_1}$.
\end{proof}

We will say that a subset of a vector space with seminorm $\nu$ is $\nu$-totally bounded if it is totally bounded with respect to $\nu$.

\begin{lem}\label{lem:appr_tot_bounded}
    Let $\nu$ be some seminorm on a normed space $\cN$ such that $\nu(x)\le||x||$ for any $x\in\cN$, and let $S$ be a bounded subset of $\cN$. Suppose that for any $\e>0$ there exists an $\e$-close $\nu$-totally bounded set $S_\e$
    (i.e., $\dist(x,S_\e)<\e$ for any $x\in S$). Then $S$ is $\nu$-totally bounded.
\end{lem} 

\begin{proof}
Fix any $\e>0$.
Note that for any $x\in S$ there exists $y\in S_\e$ such that $||x-y||<2\e$ and hence $\nu(x-y)<2\e$.

Let $\{y_1,\dots,y_D\}$ be a finite $\e$-net for $S_\e$ with respect to $\nu$. 
Let $\{y_1,\dots,y_d\}$ be the subfamily consisting of those $y_j$ for which there exists $x\in S$
(denote it by $x_j$) such that $\nu(x_j-y_j)<3\e$ (this set is not empty because for $x\in S$ we can find $y\in S_\e$ with $\nu(x-y)<2\e$, and for $y$ we can find $y_j\in\{y_1,\dots,y_D\}$ with $\nu(y-y_j)\le\e$).

Now consider $\{x_1,\dots,x_d\}$. For any $x\in S$ we can find $y\in S_\e$ with $||x-y||<2\e$, and hence $\nu(x-y)<2\e$. For this $y$ we can find $y_j\in\{y_1,\dots,y_D\}$ with $\nu(y-y_j)\le\e$, so $\nu(x-y_j)<3\e$ and hence $y_j\in\{y_1,\dots,y_d\}$ and $\nu(y_j-x_j)<3\e$.

Finally, $\nu(x-x_j)\le \nu(x-y_j) +\nu(x_j-y_j)<6\e$, i.e. $\{x_1,\dots,x_d\}$ is a finite $6\e$-net for $S$ with respect to $\nu$.
\end{proof}

\begin{cor}\label{cor:appr_tot_bounded}
    Let $S$ be a bounded subset of a Hilbert $C^*$-module $\cN$. Suppose that for any $\e>0$ there exists an $\e$-close $(\cN,\cN)$-totally bounded set $S_\e$. Then $S$ is $(\cN,\cN)$-totally bounded.
\end{cor} 

\begin{proof}
It is sufficient to apply the previous Lemma for any seminorm $\nu_{X,\F}$ with some admissible $X$ and a sequence of states $\F$ and note that $\nu_{X,\F}(x)\le||x||$ for any $x\in\cN$.
\end{proof}

The following lemma is a generalization of \cite[Lemma 4.4]{Troitsky2020JMAA}.

\begin{lem}\label{lem:totb_acomp}
    If some subset $T$ of $\ell_2(\A)=\bigoplus\limits_{k\in\mathbb N}\A$ is $(\ell_2(\A),\ell_2(\A))$-totally bounded, then for any $\e>0$ there exists a sufficiently large integer $D$
 such that the orthogonal
    projection $Q_D:\ell_2(\A)\to\ell_2(\A)$ onto the first $D$ coordinates satisfies
    $$
    ||x-Q_Dx||<\e
    $$
for any $x\in T$. If $\A$ is unital, this means that $T$ is $\A$-precompact in the sense of \cite{m_bratisl}.  
\end{lem}
    
\begin{proof}
Suppose that this is not true.
    Then there exist $\delta>0$, a sequence of elements $t_i\in T$, and a corresponding
increasing sequence of numbers $j(i)\to\infty$ such that 
$\|(1-Q_{j(i)}) t_i\|>\delta$. 
Without loss of generality  
(passing to a subsequence of $j(i)$ if necessary)
we can assume $\|(Q_{j(i+1)}-Q_{j(i)}) t_i\|>\frac{3}{4}\delta$. Denote $\hat q_{j(i)}:=Q_{j(i+1)}-Q_{j(i)}$.
Let 
$\mu_i\in L_{j(i+1)}\ominus L_{j(i)}=\Im \hat q_{j(i)}$ (where $L_{j(i)}=\Im Q_{j(i)}$)
be an element of norm 1 such that
\begin{equation}\label{eq:ozenkasomegojj}
\| \< \mu_i, \hat q_{j(i)} t_i\>\|>\frac{3}{4}\delta 
 \end{equation} 
 (one can take $\mu_i= \hat q_{j(i)} t_i/
 \|\hat q_{j(i)} t_i\|$).
 Define $x_i\in \ell_2(\A)$ to be $\mu_i$ extended by zero on the orthogonal complement to $L_{j(i+1)}\ominus L_{j(i)}$. In other words,
 $$
 x_i = \hat q_{j(i)}^* \mu_i, \qquad (1-\hat q_{j(i)}) x_i =0.
 $$
For each $i=1,2,\dots$, there exists a state $\f_i$ (see \cite[Lemma 1.2]{Troitsky2020JMAA}, for example) such that
\begin{equation}\label{eq:bolshiedalekieell}
|\f_i (\<\mu_i, \hat q_{j(i)} t_i\>)|>\frac{3\delta}{8}.
\end{equation}
For these data there exists
$d_{X,\F}$ with $X=\{x_i\}$ and $\F=\{\f_i\}$, because
$X$ is evidently admissible for $\ell_2(\A)$.
Thus, one can find a
finite collection $\{y_1,\dots,y_n\}$ of elements of $T$
such that for any $y\in T$ there exists $k\in \{1,\dots,n\}$
with
\begin{equation}\label{eq:ozenkadliasetiell}
d_{X,\F}(y,y_k)<\frac{\delta}{8}.
\end{equation}
There exists a sufficiently large $D$ such that
\begin{equation}\label{eq:ubyvaniesetiell}
\|(1-Q_D) y_k\|<\frac{\delta}{8},\qquad k=1,\dots,n.
\end{equation}
Then for $i$ with $j(i)>D$ and any $k\in \{1,\dots,n\}$, 
by (\ref{eq:bolshiedalekieell}) and (\ref{eq:ubyvaniesetiell})
we have
$$
d_{X,\F} (t_i,y_k)\ge |\f_i(\<t_i-y_k,x_i\>)|=|\f_i(\<x_i,t_i-y_k\>)|
$$
$$
=|\f_i(\<\mu_i, \hat q_{j(i)} t_i\>) -\f_i(\<\mu_i, \hat q_{j(i)} y_k\>)|
\ge \frac{3\delta}{8} - \frac{\delta}{8} = \frac{\delta}{4}.
$$
This contradicts (\ref{eq:ozenkadliasetiell}).
    \end{proof}

\section{Main results}

\begin{prop}
    
\label{prop:compcount}
Let $\cN$ be a countably generated Hilbert $C^*$-module and let
$Z\subset \cN$ be a bounded 
subset such that for every $\varepsilon>0$ there exist $M_\varepsilon>0$ and elements $z_1,\dots,z_s\in \cN$
such that for every $x\in Z$ there exist $a_1(x),\dots,a_s(x)\in \A$ with $\|a_k(x)\|\le M_\varepsilon$ and
\[
\Bigl\|x-\sum_{k=1}^s z_k a_k(x)\Bigr\|<\varepsilon
\]
(for example, this holds when $Z$ is a priori bounded and for some $y_1,\dots,y_s$ we have $a_k(x)=\<y_k,x\>$ since $||\<y_k,x\>||\le M\max\limits_{k=1,\dots,s}||y_k||$, where $M=\sup_{x\in Z}||x||$).

Then 
for any standard frame $\{x_j\}_{j\in \mathbb N}$ of $\cN$ (which is at most countable by Lemma \ref{lem:countcountframe}
) with canonical dual $\{g_j\}_{j\in \mathbb N}$ the reconstruction formula holds uniformly in $x\in Z$, i.e. for every $\varepsilon>0$ there exists $N\in\mathbb N$ such that for all $n>N$ 

$$ \sup_{x\in Z}\Bigl\|x-\sum_{j=1}^n x_j\< g_j,x\>\Bigr\|<\varepsilon. $$

\end{prop}

\begin{proof}

Fix an arbitrary countable standard frame $\{x_j\}_{j\in\mathbb N}$ (with frame constants $c_1$, $c_2$ and canonical dual $\{g_j\}$), which exists since any countably generated Hilbert $C^*$-module has a standard frame.

Take $\e>0$, and for $\e'=\e\frac{c_1}{3c_2}$ 
take the corresponding $M_{\e'}$ and $z_1,\dots,z_s$, i.e. for 
any $x\in Z$ there exist $a_1(x),\dots,a_s(x)\in\A$, $||a_k(x)||\le M_{\e'}$ for $k=1,\dots,s$, such that 
\begin{equation}\label{eq:forak}
\left\|x-\sum_{k=1}^s z_k a_k(x)\right\|<\e'=\e\frac{c_1}{3c_2}.
\end{equation}
Note that $\frac{c_1}{c_2}\le1$.

From the reconstruction formula for $z_k$, $k=1,\dots,s$ there exists $N\in\mathbb N$ such that, for any $n>N$,
$$
\left\|z_k-\sum\limits_{j=1}^nx_j\<g_j,z_k\>\right\|\le\frac{\e}{3sM_{\e'}}, \qquad   k=1,\dots,s.
$$

Then, for $n>N$ and any $x\in Z$ we have
$$
\left\|x-\sum_{j=1}^n x_j\< g_j,x\>\right\|\le
$$
\begin{multline*}
\le\left\|x-\sum_{k=1}^s z_k a_k(x)\right\|+\left\|\sum_{k=1}^s z_k a_k(x)-\sum_{k=1}^s\sum_{j=1}^nx_j\<g_j, z_k a_k(x)\>\right\|
\\
+\left\|\sum_{j=1}^n\sum_{k=1}^sx_j\<g_j, z_k a_k(x)\>-\sum_{j=1}^n x_j\< g_j,x\>\right\|\le
\end{multline*}
$$\le
\e\frac{c_1}{3c_2}+\sum\limits_{k=1}^s\left\|\left(z_k-\sum\limits_{j=1}^nx_j\<g_j,z_k\>\right)a_k(x)\right\|
+\left\|\sum\limits_{j=1}^nx_j\left\<g_j,\left(\sum\limits_{k=1}^sz_ka_k(x)\right)-x\right\>\right\|\le
$$
$$\le\frac{\e}{3}+s\frac{\e}{3sM_{\e'}}M_{\e'}+\frac{c_2}{c_1}\e\frac{c_1}{3c_2}\le\e,
$$
where the last term is estimated using  Lemma \ref{lem:partsum}
and (\ref{eq:forak}).
\end{proof}

\begin{teo}\label{thm:setcomp}
Let $Z\subset \cN$ be a bounded subset. The following statements are equivalent.

\begin{enumerate}
\item[\rm a)] For every $\varepsilon>0$ there exist $M_\varepsilon>0$ and elements $z_1,\dots,z_s\in \cN$
such that for every $x\in Z$ there exist $a_1,\dots,a_s\in \A$ with $\|a_k\|\le M_\varepsilon$ and
\[
\Bigl\|x-\sum_{k=1}^s z_k a_k\Bigr\|<\varepsilon.
\]

\item[\rm b)] 
There exists a countably generated Hilbert $\A$-submodule $\cN^Z\subseteq\cN$ containing $Z$ such that for any standard frame $\{x_j\}_{j\in \mathbb N}$ of $\cN^Z$ with canonical dual $\{g_j\}_{j\in \mathbb N}$ the reconstruction formula holds uniformly in $x\in Z$, i.e. for every $\varepsilon>0$ there exists $N\in\mathbb N$ such that for all $n>N$ 

$$ \sup_{x\in Z}\Bigl\|x-\sum_{j=1}^n x_j\< g_j,x\>\Bigr\|<\varepsilon. 
$$
Moreover, this holds for any countably generated submodule $\cN^0\subseteq\cN$ with $\cN^Z\subseteq\cN^0$ and for any standard frame $\{x_j\}_{j\in \mathbb N}$ of $\cN^0$ with canonical dual $\{g_j\}_{j\in \mathbb N}$.

\item[\rm c)] 
The restriction of the identity operator on $Z$ is (informally speaking) $\A$-compact, i.e.
(formally speaking)
for every $\varepsilon>0$ there exists a finite $\A$-rank operator 
$T$, $Tx=\sum\limits_{k=1}^nz_k\<y_k,x\>$, $z_k,y_k\in\cN$ such that
\[
\sup_{x\in Z}\|x-Tx\|<\varepsilon.
\]

\item[\rm d)] 
The restriction of the identity operator on $Z$ is Banach-compact, i.e.
for every $\varepsilon>0$ there exists a finite Banach-rank operator $T$, $Tx=\sum\limits_{k=1}^nz_kf_k(x)$, $z_k\in\cN$, $f_k\in\cN'$, such that
\[
\sup_{x\in Z}\|x-Tx\|<\varepsilon.
\]
\end{enumerate}
\end{teo}
\begin{dfn}\label{dfn:setcomp}
We call a bounded subset satisfying any of these conditions \emph{$\A$-pre\-com\-pact.}
\end{dfn}

\begin{rk}
    Note that condition \(a)\) automatically implies that \(Z\) is bounded (indeed, for $\e=1$ we have $||x||\le||x-\sum_{k=1}^s z_k a_k||+||\sum_{k=1}^s z_k a_k||\le1+s M_1\max\limits_{k=1,\dots,s}\{||z_k||\}$).

    Condition $b)$ also implies boundedness. Let
\(\{u_j\}_{j\in\mathbb N}\) be a countable Parseval frame in the
submodule \(\mathcal N^Z\) occurring in $b)$ (which exists by \cite[Theorem 5.3]{FrankLarson2002}). Then
\[
y_{m,j}:=2^{-m/2}u_j,\qquad (m,j)\in\mathbb N\times\mathbb N,\qquad m\ge1,
\]
is again a Parseval frame, because for every \(x\in\mathcal N^Z\)
$$
\sum_{m=1}^{\infty}\sum_{j=1}^{\infty} \langle x,y_{m,j}\rangle\langle y_{m,j},x\rangle=\sum_{m=1}^{\infty}2^{-m} \sum_{j=1}^{\infty}\langle x,u_j\rangle\langle u_j,x\rangle=
\sum_{m=1}^{\infty}2^{-m}\langle x,x\rangle=\langle x,x\rangle.
$$

We can enumerate this frame as a sequence in an arbitrary way.
Condition $b)$ states that for $\e=1$ we can find an initial segment $F\subset\mathbb N\times\mathbb N$ of this enumeration such that 
$$ \sup_{x\in Z}\Bigl\|x-\sum_{(m,j)\in F} \theta_{y_{m,j},y_{m,j}}(x)\Bigr\|<1. 
$$

Denote $P_F=\sum\limits_{(m,j)\in F} \theta_{y_{m,j},y_{m,j}}$ and $r=\max\limits_{}\{m:(m,j)\in F\}$.
For every \(x\in\mathcal N^Z\) we have
\[
\langle P_Fx,x\rangle=
\sum_{(m,j)\in F}\langle x,y_{m,j}\rangle\langle y_{m,j},x\rangle\le
\sum_{m=1}^{r}2^{-m}\sum_{j=1}^{\infty}
\langle x,u_j\rangle\langle u_j,x\rangle=(1-2^{-r})\langle x,x\rangle.
\]
Therefore
$$
\langle (I-P_F)x,x\rangle\ge2^{-r}\langle x,x\rangle
$$
and
$$
||(I-P_F)x||\cdot||x||\ge||\<(I-P_F)x,x\>||\ge||2^{-r}\<x,x\>||=2^{-r}||x||^2,
$$
hence
\[
\|(I-P_F)x\|\geq 2^{-r}\|x\|,\qquad x\in\mathcal N^Z,
\]
and for $x\in Z$
$$
||x||\le2^r\sup_{x\in Z}\Bigl\|x-\sum_{(m,j)\in F} \theta_{y_{m,j},y_{m,j}}(x)\Bigr\|<2^r.
$$
Thus $Z$ is bounded by $2^r$.

On the other hand, conditions \(c)\) and \(d)\) do not imply boundedness,
because $Z$ may contain some finitely generated submodule and partial sums may act identically on it. 
    So without the requirement of boundedness $a)$ and $b)$ are not equivalent to the rest.
\end{rk}

\begin{proof}[Proof of Theorem \ref{thm:setcomp}]
$a)\Rightarrow b)$. 

For $\e_n=\frac{1}{n}$, $n\in\mathbb N$, consider a countable family of finite sets $\{z^{(n)}_1,\dots,z^{(n)}_{s_n}\}$ with corresponding constants $M_{1/n}$. Thus, for any $\e>0$ there exists $\e_n<\e$ such that for any $x\in Z$ there exists $a_1(x),\dots,a_{s_n}(x)\in \A$ with 
$||x-\sum_{k=1}^{s_n} z^{(n)}_k a_k(x)||<\frac{1}{n}<\e$, so $Z$ lies in the countably generated 
submodule $\cN^Z$,

$$\cN^Z=\overline{\Span_\A\{z^{(n)}_k,n\in\mathbb N,k=1,\dots,s_n\}}.$$

Then we can apply Proposition \ref{prop:compcount} to any countably generated submodule $\cN^0\subseteq\cN$ with $\cN^Z\subseteq\cN^0$.

Thus, $b)$ holds.

The implications $b)\Rightarrow c)\Rightarrow d)$ are obvious. Let us prove $d)\Rightarrow a)$.

Fix $\varepsilon>0$ and choose $T=\sum\limits_{k=1}^sz_kf_k(\cdot)$, $f_k\in\cN'$, with $\sup_{x\in Z}\|x-Tx\|<\varepsilon$.
Set $a_k(x):=f_k(x)\in \A$. Since $Z$ is bounded, $R:=\sup_{x\in Z}\|x\|<\infty$ and
\[
\|a_k(x)\|=\|f_k(x)\|\le \|f_k\|\cdot\|x\|\le \|f_k\|\cdot R.
\]
Thus with $M_\varepsilon:=R\max\limits_{k=1,\dots,s}\|f_k\|$ we have $\|a_k(x)\|\le M_\varepsilon$ for all $x\in Z$,
and
\[
\Bigl\|x-\sum_{k=1}^s z_k a_k(x)\Bigr\|=\|x-Tx\|<\varepsilon.
\]
This proves a).
\end{proof}

\begin{rk}\label{rk:lambda}
Note that condition a) is a reformulation of the condition $\lambda(Z)=0$ for the measure of noncompactness constructed in \cite{KeckicLazovic2026}, while the remaining conditions are conceptually related to the approximation property, which was systematically developed by Grothendieck (for example, in \cite{Grot1952} and \cite{Grot1955}).
\end{rk}

\begin{rk}\label{rk:from_count}
The following theorem is in fact a generalization of 
the equivalence between \(\lambda(Z)=0\) and total boundedness from
\cite[Theorem 3.23]{KeckicLazovic2026}, which was formulated in terms of the measure of noncompactness and where countably generated modules were considered.
\end{rk}

\begin{teo}\label{thm:totb_comp}
Let $Z\subset \cN$ be a subset. If $Z$ is $\A$-precompact, then it is $(\cN,\cN)$-totally bounded.

If, in addition, $\cN$ admits a standard frame (for example, if $\cN$ is countably generated), then the converse is also true: if $Z$ is $(\cN,\cN)$-totally bounded, then it is $\A$-precompact.
\end{teo}

\begin{proof}
    Let $Z$ be $\A$-precompact. Then for any $\e>0$ there exists an operator $T=\sum\limits_{k=1}^sz_k\<y_k,x\>$ such that
\[
\sup_{x\in Z}\|x-Tx\|<\varepsilon.
\]
    Thus, $Z$ is approximated in the sense of Lemma \ref{lem:appr_tot_bounded} 
    by a set $T(Z)$ which is $(\cN,\cN)$-totally bounded since it is contained in the image $T(B_R)$ of the ball of radius $R$ such that $Z\subset B_R$, and $T(B_R)$ is $(\cN,\cN)$-totally bounded by \cite[Theorem 2.4]{TroitFuf2020}.
    So by Corollary \ref{cor:appr_tot_bounded}, $Z$ is $(\cN,\cN)$-totally bounded too.

Suppose that \(Z\) is \((\mathcal N,\mathcal N)\)-totally bounded and \(\mathcal N\) admits a standard frame \(\{x_j\}_{j\in J}\). Let
\[
\Theta\colon\mathcal N\longrightarrow\ell_2^J(\mathcal A), \qquad\Theta(x)=(\<x_j,x\>)_{j\in J}
\]
be its analysis operator, let
\[
S=\Theta^*\Theta, \qquad S(x)=\sum\limits_{j\in J}x_j \<x_j,x\>,
\]
be the frame operator, and put
\[
V=\Theta S^{-1/2}.
\]
Then
\[
V^*V=S^{-1/2}\Theta^*\Theta S^{-1/2}=I_{\mathcal N},
\]
so \(V\) is an adjointable isometry and \(V(\mathcal N)\) is an orthogonally complemented submodule of \(\ell_2^J(\mathcal A)\).
Also \(V(Z)\) is \((V(\mathcal N),V(\mathcal N))\)-totally bounded because $V$ is an isometry.

Since
\[
\ell_2^J(\mathcal A)=V(\mathcal N)\oplus V(\mathcal N)^\perp,
\]
Lemma~\ref{lem:directsum_totbu} implies that \(V(Z)\) is
$\bigl(\ell_2^J(\mathcal A),\ell_2^J(\mathcal A)\bigr)\text{-totally bounded}$.
By \cite[Lemma 3.3]{TroitFuf2020}, there exists a countable subset $\widetilde J=\{j_1,j_2,\ldots\}\subseteq J$ such that
\[
V(Z)\subseteq\ell_2^{\widetilde J}(\mathcal A)
\]
and \(V(Z)\) is $(\ell_2^{\widetilde J}(\mathcal A),\ell_2^{\widetilde J}(\mathcal A))$-totally bounded.

Let \(Q_D\) be the coordinate projection onto the coordinates \(j_1,\ldots,j_D\). By Lemma~\ref{lem:totb_acomp}, for every \(\varepsilon>0\) there exists \(D\) such that
\[
\sup_{x\in Z}\|(I-Q_D)Vx\|=\sup_{x\in Z}\|Vx-Q_DVx\|<\varepsilon.
\]

Put $u_j=S^{-1/2}x_j$ (note that $S^{-1/2}$ is self-adjoint) and consider $P_D:=\sum\limits_{k=1}^D\theta_{u_{j_k},u_{j_k}}$, $P_D(x)=\sum\limits_{k=1}^Du_{j_k}\langle u_{j_k},x\rangle=S^{-1/2}\sum\limits_{k=1}^Dx_{j_k}\langle x_{j_k},S^{-1/2}x\rangle=S^{-1/2}\circ\Theta^*\circ Q_D\circ\Theta\circ S^{-1/2}(x)=V^*\circ Q_D\circ V(x)$.

Note that for any $x\in Z$ we have $x=V^*Vx$, so $||x-P_Dx||=||V^*Vx-V^* Q_D Vx||\le||Vx-Q_D Vx||<\e$, so the identity operator is approximated uniformly on $Z$ by operators of finite $\A$-rank, 
and by \cite[Corollary 1.26]{FufTro2024UniformStr} $Z$ is bounded,
so $Z$ is $\A$-precompact.

\end{proof}

One can describe Banach-compact operators as bounded $\A$-operators such that the image of the unit ball can be uniformly approximated by elements of a finitely generated $\A$-module. The same is true for $\A$-compact operators.
More precisely, we have the following statement.

\begin{teo}\label{thm:opercomp}
Let $F: \M\to \cN$ be a bounded (resp., adjointable) $\A$-linear operator and let $B$ be the unit ball of $\M$.
Then the following are equivalent:
\begin{enumerate}
\item[\rm (a)] $F\in \bBK(\M,\cN)$, i.e. it is Banach-compact (resp. $\A$-compact).
\item[\rm (b)] $F(B)$ is $\A$-precompact in $\cN$.
\end{enumerate}
\end{teo}

\begin{proof}
Suppose that $F$ is Banach-compact.
Fix an $\e>0$. Then, for some elements $z_k \in \cN$ and functionals $f_k\in \M'$, $k=1,\dots,s$, one has
    $\|F(x)-\sum\limits_{k=1}^s z_k f_k(x)\|<\e$ for any $x\in B$. 

For any $t\in F(B)$ one can find some $x\in B$ (not necessarily unique) such that $t=F(x)$. So we can take $a_k(t)=f_k(x)$
and note that $||a_k(t)||\le||f_k||\cdot||x||\le||f_k||$, so we can take $M_\e=\max\limits_{1\le k\le s}||f_k||$.
Then $||t-\sum\limits_{k=1}^s z_k a_k(t)||=||F(x)-\sum\limits_{k=1}^s z_k f_k(x)||<\e$, so $F(B)$ is $\A$-precompact.
For $\A$-compact $F$ the same argument applies with functionals of the special form $f_i(x)=\<y_i,x\>$. 

Conversely, if $F$ is a bounded morphism and $F(B)$ is $\A$-precompact, then for any $\e>0$ there exists an operator $T=\sum\limits_{k=1}^n\theta_{z_k,y_k}$, $z_k,y_k\in\cN$ such that $||F(x)-TF(x)||<\e$ for any $x\in B$ and
    
    $$
    T F(x)= \sum\limits_{k=1}^nz_k\<y_k,F(x)\>=\sum\limits_{k=1}^nz_kf_k(x)
    $$
with $f_k(x)=\<y_k,F(x)\>$.
It is a finite Banach-rank operator $\M\to \cN$, so $F$ is Banach-compact. 
In the case of an adjointable $F$ we can rewrite 
$\sum\limits_{k=1}^nz_k\<y_k,F(x)\>=\sum\limits_{k=1}^nz_k\<F^*y_k,x\>$, so $TF$ is a finite $\A$-rank operator.
\end{proof}

\begin{cor}\label{cor:adj_acomp}
    A Banach-compact operator is an $\A$-compact operator if and only if it is adjointable.
\end{cor}

\begin{cor}\label{cor:totb_comp}
Let $F: \M\to \cN$ be a bounded (resp., adjointable) $\A$-linear operator. If $F\in \bBK(\M,\cN)$ (resp., $\A$-compact), then $F(B)$ is $(\cN,\cN)$-totally bounded.

If, in addition, $\cN$ admits a standard frame (for example, if $\cN$ is countably generated), then the converse is also true: if $F(B)$ is $(\cN,\cN)$-totally bounded, then $F\in \bBK(\M,\cN)$ (resp., $\A$-compact).
\end{cor}

We have one more result on the approximation of Banach-compact operators as a corollary.

\begin{cor}\label{cor_series}
Suppose that $T\in\bBK(\M,\cN)$ and $B$ is the unit ball of $\M$.
Then there exists a countably generated submodule $\cN^{T(B)}\subseteq\cN$ such that ${\operatorname{Ran}(T)}\subseteq\cN^{T(B)}$
and $T$ can be approximated in operator norm by a combination of $\theta_{x_j,f}$ where $x_j$ are elements of an arbitrary (countable) standard frame in 
$\cN^0\subseteq\cN$, where $\cN^0$ is an arbitrary countably generated submodule with $\cN^{T(B)}\subseteq\cN^0$. Moreover, it can be approximated by a combination of $\theta_{x_j,\widetilde{g_j}\circ T}$, where $\{g_j\}$ is the canonical dual frame of $\{x_j\}$ and $\widetilde{g_j}\in\cN'$, $\widetilde{g_j}(t)=\<g_j,t\>$. In other words, $T=\sum\limits_{j=1}^\infty\theta_{x_j,\widetilde{g_j}\circ T}$, where the series converges in norm.
\end{cor}

\begin{proof}
Since $T$ is Banach-compact, $T(B)\subset\cN$ is $\A$-precompact, so 
$T(B)$ (and hence ${\operatorname{Ran}(T)}$) lies in some countably generated submodule $\cN^{T(B)}$ such that for any countably generated submodule
$\cN^0$ with $\cN^{T(B)}\subseteq\cN^0$ and any countable standard frame $\{x_j\}$ of $\cN^0$ with canonical dual $\{g_j\}$ and
for any $\e>0$ there exists $N\in\mathbb N$ such that for any $n>N$ we have
$||Tx-\sum\limits_{j=1}^n x_j\< g_j,Tx\>||<\e$ for all $x\in \M$, $||x||\le1$.
Note that $\< g_j,Tx\>=\widetilde{g_j}\circ T(x)$, so 
$||T-\sum\limits_{j=1}^n \theta_{x_j,\widetilde{g_j}\circ T}||<\e$.
\end{proof}

\section{Particular cases and examples}

In \cite{m_bratisl} the notion of an $\A$-precompact set was already defined for a standard module over a unital $C^*$-algebra, where elements $z_1,\dots,z_s$ were to be orthogonal generators of a free module. This is a strong condition at least because it cannot hold for modules over nonunital algebras. To work in a more general situation, we must remove the condition that the generators are free, but 
in this case the remaining condition is not sufficient for the compactness criterion for operators, even in the case of a standard module over a unital algebra, as the example in Theorem \ref{ex_bad} shows. 
This example also illustrates an essential difference between topologically finitely generated modules and algebraically finitely generated modules studied in \cite{BlecherKania2014Studia}.
But for free generators, the condition that the coefficients are bounded by some constant $M_\e$ automatically holds (as we will prove in the next theorem), so this definition of $\A$-precompactness is indeed a generalization.

But modules over unital algebras can nevertheless be complicated and, in particular, have no non-singular elements (for example, we can consider any $\A$-module as a module over the unitalization algebra $\dot\A$), so to obtain the equivalence (i.e. the converse implication), we restrict ourselves to the case of the standard module.

\begin{teo}\label{thm_old}
    Let $\A$ be a unital $C^*$-algebra, $\cN$ a Hilbert $\A$-module, and $Z$ a bounded subset of $\cN$.
    
If for every $\varepsilon>0$ there exists
a free finitely generated submodule $L_s\subset\cN$
such that $\dist(Z,L_s)<\e$, then $Z$ is $\A$-precompact in the sense of Definition \ref{dfn:setcomp}.

If, in addition, $\cN=\ell_2^{J}(\A)$ for some index set $J$, then the converse is also true.
\end{teo}

\begin{proof}
Fix $\e>0$ and take the corresponding submodule $L_s$. Denote by $e_1,\dots,e_s$ some free generators of $L_s$, i.e. $\<e_i,e_j\>=\delta_{ij}$. Define the operator $P_{L_s}$ by the formula $P_{L_s}x=\sum\limits_{j=1}^se_j\<e_j,x\>\in L_s$.

For any $x\in Z$ we can find $y\in L_s$ such that $||x-y||<\e$ (note that $P_{L_s}y=y$), so 

$$
||x-P_{L_s}x||\le||x-y||+||P_{L_s}y-P_{L_s}x||<\e+||P_{L_s}||\cdot||y-x||<2\e,
$$
i.e. the identity operator is approximated on $Z$ by operators of finite $\A$-rank, so $Z$ is $\A$-precompact.

Now suppose that $\cN=\ell_2^{J}(\A)$ and $Z$ is $\A$-precompact. In particular, $Z$ lies in some countably generated submodule $\cN^Z$; hence, $Z$ and $\cN^Z$ lie in some submodule $\ell_2^{\widetilde{J}}(\A)$ for some countable $\widetilde{J}\subset J$ (since any generator of $\cN^Z$ has at most countably many nonzero coordinates). This submodule has a countable standard frame consisting of the standard basis $e_j$, $j\in \mathbb N$, which is a normalized tight frame, so it coincides with its canonical dual. Since $Z$ is $\A$-precompact, for any $\e>0$ there exists $s\in\mathbb N$ such that $||x-\sum\limits_{j=1}^se_j\<e_j,x\>||<\e$ for any $x\in Z$, i.e. $\dist(Z,L_s)<\e$, where $L_s=\Span(e_1,\dots,e_s)$ is a free module generated by $e_1,\dots,e_s$.
\end{proof}

If the $C^*$-algebra $\A$ is finite-dimensional, then, as in the case of Hilbert spaces, the condition that coefficients can be chosen to be bounded uniformly holds automatically.

\begin{teo}
    Let $\A$ be a finite-dimensional $C^*$-algebra, $\cN$ a Hilbert $\A$-module, $Z$ a bounded subset of $\cN$.

    \begin{enumerate}        
\item[\rm 1)] If for some $\varepsilon>0$ there exist elements $z_1,\dots,z_s\in \cN$
such that for every $x\in Z$ one can find $a_1,\dots,a_s\in \A$ with
\[
\Bigl\|x-\sum_{k=1}^s z_k a_k\Bigr\|<\varepsilon,
\]
then
there exists $M_\varepsilon>0$ 
such that $a_1,\dots,a_s$ can be chosen so that $\|a_k\|\le M_\varepsilon$.  

\item[\rm 2)] If for any $\varepsilon>0$ there exist elements $z_1,\dots,z_s\in \cN$
such that for every $x\in Z$ one can find $a_1,\dots,a_s\in \A$ with
\[
\Bigl\|x-\sum_{k=1}^s z_k a_k\Bigr\|<\varepsilon,
\]
then $Z$ is $\A$-precompact in the sense of Definition \ref{dfn:setcomp}.
    \end{enumerate}
\end{teo}

\begin{proof}
    Obviously, 2) follows from 1), so it is enough to prove 1).

Fix some $\e>0$ and the corresponding $z_1,\dots,z_s\in \cN$.
For any $x\in Z$ we can find $y_x=\sum\limits_{k=1}^s z_k a_k$ such that $||x-\sum\limits_{k=1}^s z_k a_k||<\e$. The set $Y=\{y_x\}_{x\in Z}$ is also bounded, $||y_x||<D$ for some $D>0$.

Consider a linear operator $G:\A^s\to\cN$, $G(a_1,\dots,a_s)=\sum\limits_{k=1}^s z_k a_k$. Then $G$ admits the adjoint $G^*(x)=(\<z_1,x\>,\dots,\<z_s,x\>)$, because
$$
\left\<\sum\limits_{k=1}^s z_k a_k,y\right\>=\sum\limits_{k=1}^s a_k^*
\<z_k ,y\>.
$$
The range $\operatorname{Ran}(G)$ is a finite-dimensional vector space, hence a closed submodule. Therefore $\Ker(G)$ is an orthogonally complemented submodule: for $L=(\Ker(G))^\bot$, one has $\Ker(G)\oplus L=\A^s$ (see \cite[Theorem 2.3.3]{MTBook}).
So $G|_L:L\to\operatorname{Ran}(G)$ is an isomorphism, and 
there exists a bounded morphism $(G|_L)^{-1}:\operatorname{Ran}(G)\to L$, $||(G|_L)^{-1}||<B$ for some $B>0$.

Thus, for any $y_x\in Y$ and $(a_1,\dots,a_s)=(G|_L)^{-1}(y_x)$,  we have $||(a_1,\dots,a_s)||\le B\cdot D$. Hence, $||a_k||<B\cdot D$. Take $M_\e=B\cdot D$.

Thus, $||x-G((G|_L)^{-1}(y_x))||<\e$, and $G((G|_L)^{-1}(y_x))=\sum\limits_{k=1}^s z_k a_k$ for some $a_1,\dots,a_s\in \A$ with $\|a_k\|\le M_\varepsilon$.
\end{proof}

\begin{rk}\label{rk:partcase}
The following theorem deals with a particular case of the construction in \cite[Example 3.13]{KeckicLazovic2026}.
\end{rk}

\begin{teo}\label{ex_bad}
    There exists a $C^*$-algebra $\A$, a Hilbert $C^*$-module $\cN$ and an operator $F:\cN\to\cN$ such that for
every $\varepsilon>0$ there exist elements $z_1,\dots,z_s\in \cN$
such that for every $x\in F(B)$ there exist $a_1,\dots,a_s\in \A$ with
\[
\Bigl\|x-\sum_{k=1}^s z_k a_k\Bigr\|<\varepsilon,
\]
but $F(B)$ is not $\A$-precompact and hence $F$ is not Banach-compact.
\end{teo}

\begin{proof}
Take $\A=c$, the algebra of all convergent sequences with the norm $||a(\cdot)||=\sup\limits_{n\in\mathbb N}{|a(n)|}$, $c\cong C(\mathbb N^+)$, where $\mathbb N^+$ is the one-point compactification of the positive integers.
Let $\cN$ be equal to $\ell_2(c)$ --- the standard module over $c$. Note that this module has a standard set of generators, i.e. the standard orthonormal basis $\{e_j\}_{j\in\mathbb N}$, $e_j=(0,\dots,0,1_\A,0,\dots)$, where 
all coordinates except the $j$-th are zero, and the $j$-th coordinate is the identity of $c$, i.e. the sequence identically equal to one.
Thus, for any element $x\in\cN$ we can write $x=\sum\limits_{k\in\mathbb N}e_ka^k$, where $a^k\in c$.
Also take $\d^j\in c$ such that $\d^j(j)=1$, $\d^j(i)=0$ for $i\ne j$.

Then define $F(x)=\sum\limits_{k\in\mathbb N}F(e_k)a^k=\sum\limits_{k\in\mathbb N}e_k\d^ka^k$. 

Note that for $x=\sum\limits_{k\in\mathbb N}e_ka^k$, $y=\sum\limits_{k\in\mathbb N}e_kb^k\in\cN$ we have 
$$\<F(x),y\>=
\<\sum\limits_{k\in\mathbb N}e_k\delta^ka^k,\sum\limits_{k\in\mathbb N}e_kb^k\>=
\sum\limits_{k\in\mathbb N}\delta^k(a^k)^*b^k=
\<\sum\limits_{k\in\mathbb N}e_ka^k,\sum\limits_{k\in\mathbb N}e_k\delta^kb^k\>=
\<x,F(y)\>,$$
so $F$ is adjointable (in fact, self-adjoint), and obviously it is a projection.

Let us prove that any element of $F(B)$ can be approximated by 
$\A$-multiples of
just one element $v\in F(B)$, 
$v=\sum\limits_{j\in\mathbb N}\frac{1}{j!}e_j\d^j$ (note that $v\d^k=\frac{1}{k!}e_k\d^k$ and $F(v)=v$). 

Indeed, fix $\e>0$ and $y\in F(B)$, $y=\sum\limits_{k\in\mathbb N}e_k\d^ka^k=F(\sum\limits_{k\in\mathbb N}e_ka^k)=F(x)$ for some $x\in B$. In particular, there exists $N\in\mathbb N$ such that 
$$
\left\|\sum\limits_{k=N+1}^\infty e_ka^k\right\|_{\cN}=\sqrt{\left\|\sum\limits_{k=N+1}^\infty (a^k)^*a^k\right\|_\A}=\sqrt{\sup\limits_{n\in\mathbb N}\sum\limits_{k=N+1}^\infty|a^k(n)|^2}<\e.
$$
Then for $b_{y,\e}=\sum\limits_{k=1}^Nk!\d^ka^k\in c$
\begin{multline*}
\left\|y-vb_{y,\e}\right\|_\cN=
\left\|y-\sum\limits_{k=1}^Nv\d^ka^kk!\right\|_\cN=
\left\|\sum\limits_{k=N+1}^\infty e_k\d^ka^k\right\|_\cN
=\\=\sqrt{\sup\limits_{n\in\mathbb N}\sum\limits_{k=N+1}^\infty\d^k(n)|a^k(n)|^2}\le
\sqrt{\sup\limits_{n\in\mathbb N}\sum\limits_{k=N+1}^\infty|a^k(n)|^2}
<\e.
\end{multline*}
In particular, $F(\cN)$ is generated by just one element $v$ (topologically, i.e. $F(\cN)=\overline{\Span_\A(v)}$, not algebraically). 

Note that for $x=\sum_{k=1}^\infty e_k a^k\in \cN$ we have

\begin{multline*}
\|F(x)\|_{\cN}^2=
\left\|\sum_{k=1}^\infty(\delta^k a^k)^*(\delta^k a^k)\right\|_{\A}
=\sup_{n\in\mathbb N} \sum\limits_{k=1}^\infty \delta^k(n)|a^k(n)|^2 
=\sup_{n\in\mathbb N}|a^n(n)|^2 \leq\\
\leq\sup_{n\in\mathbb N}\sum_{k\geq 1}|a^k(n)|^2 =
\left\|\sum_{k\geq 1}(a^k)^*a^k\right\|  
=\|x\|^2.
\end{multline*}

Therefore \(\|F(x)\|\leq \|x\|\) and so $\|F\|\leq 1$.
Moreover, 
$$
F(xb)=F\left(\sum\limits_{k=1}^\infty e_k(a^k b)\right)  =
\sum\limits_{k=1}^\infty e_k\delta^k a^k b  =
\left(\sum\limits_{k=1}^\infty e_k\delta^k a^k\right)b =F(x)b,
$$
so $F$ is a bounded $\A$-linear operator and $F(B)$ is a bounded set.

To verify that $F(B)$ is not $\A$-precompact, consider the orthogonal basis $\{e_n\}_{n\in\mathbb N}$ as a standard frame (which coincides with its canonical dual since it is a Parseval frame)
for $\cN$. 
For every $n\in\mathbb N$, taking $x=e_{n+1}\d^{n+1}=F(e_{n+1})$, we obtain the estimate:
$$
\sup\limits_{x\in F(B)}\Bigl\|x-\sum_{j=1}^n e_j\< e_j,x\>\Bigr\|\ge
||e_{n+1}\d^{n+1}-\sum_{j=1}^n e_j\< e_j,e_{n+1}\d^{n+1}\>||=||e_{n+1}\d^{n+1}||=1.
$$ 
Thus, this value does not tend to zero uniformly on $F(B)$, so $F(B)$ is not $\A$-precompact.
\end{proof}

\medskip
\textsc{Acknowledgments:} The authors thank the participants of the seminar ``Geometry and Groups'', especially Prof. V.~Manuilov, for valuable comments and suggestions.

\end{document}